\documentclass[11pt, a4 paper]{amsart} 
\usepackage[psamsfonts]{amssymb} 
\usepackage{amsfonts,amsmath,mathabx}

\usepackage{amsthm, amssymb}
\usepackage{amsfonts}  

\newtheorem{thmA}{Theorem}

\newtheorem{corA}[thmA]{Corollary}

\newtheorem{theorem}{Theorem}[section]

\newtheorem{proposition}[theorem]{Proposition}
\newtheorem{lemma}[theorem]{Lemma}
\newtheorem{corollary}[theorem] {Corollary}
\newtheorem{addendum}[theorem] {Addendum}
\newtheorem{question}[theorem]{Question}

\theoremstyle{remark}
\newtheorem{remark}[theorem]{Remark}
\newtheorem{remarks}[theorem]{Remarks}

\theoremstyle{definition}
\newtheorem{definition}[theorem]{Definition}

\def\Z{\mathbb Z}
\def\ns{\unlhd}
\def\ssm{\smallsetminus}

\newcommand{\core}{\operatorname{Core_{RF}}}
\newcommand{\coreRF}{\operatorname{Core_{RF}}}
\newcommand{\depth}{\operatorname{Depth_{RF}}} 
\newcommand{\depthRF}{\operatorname{Depth_{RF}}} 

\def\DL{\Delta}   
\def\Z{\mathbb Z}
\def\N{\mathbb N}

\def\G{\Gamma}
\def\GG{\mathbb G}
\def\La{\Lambda}

\def\-{\overline}
\def\wh{\widehat}

\def\T{\mathcal{T}}

\def\G{\Gamma}  
\def\g{\gamma}

\def\<{\langle}
\def\>{\rangle}

\numberwithin{equation}{section}

\begin{document}

\title{Ordinals arising as residual finiteness depths}
 
\author[Martin R. Bridson]{Martin R.~Bridson}
\address{Mathematical Institute\\
Andrew Wiles Building\\
Woodstock Road\\
Oxford, OX2 6GG}
\email{bridson@maths.ox.ac.uk}  


\begin{abstract} For every natural number $n$, there exist finitely presented groups
with residual finiteness depths $\omega\cdot n$ and $\omega\cdot n + 1$. The ordinals that
arise as the residual finiteness depth of a finitely generated group (equivalently, a countable group) 
are $0,\, 1$, the countable limit ordinals, and the successors of these limit ordinals.
\end{abstract}

\maketitle


The {\em residual  
core} $\coreRF(G)$ of a group $G$ is the normal subgroup consisting of those elements of $G$ that have
trivial image in every finite quotient of $G$, i.e. the kernel of the natural 
map from $G$ to its profinite completion $\wh{G}$. If an infinite group $G$ is residually finite, then  
its {\em residual finiteness depth} is defined as $\depthRF(G) =\omega$, because the shortest chain of finite-index subgroups
interpolating between  $G$ and  $\{1\}$ has length $\omega$. If $G$ is not residually finite, then one can continue beyond 
$\omega$ by examining the finite quotients of  $\coreRF(G)$.
In \cite{BJ}, Brody and Jankiewicz define a countable group $G$ to have  residual finiteness depth 
$\omega\cdot n$ if there is a sequence of infinite groups $G=C_0>C_1>\dots >C_{n-1}$
such that $C_{n-1}$ is residually finite and $C_{i}=\coreRF(C_{i-1})$ for $i=1,\dots,n-1$.  
They also 
define what it means for $\depthRF(G)$ to be equal to other ordinals (see Section \ref{s:1}). For example, 
$\depthRF(G) = \omega\cdot n + 1$ if, instead of $C_{n-1}$ being residually finite, $\coreRF(C_{n-1})$ is a 
non-trivial finite group. And $\depthRF(G) = \omega^2$ if there is a sequence of infinite groups $G=C_0>C_1>\dots$
such that each $C_{i}=\coreRF(C_{i-1})$ is infinite but $\core^\omega(G):=\bigcap_i C_i =1$.
If $\core^\omega(G)$ is a non-trivial finite group, then $\depth(G) = \omega^2+1$. 

\smallskip

Using a wreath product construction, Brody and Jankiewicz  
proved \cite[Theorem 1.1]{BJ} that for every  integer $n>0$ there exist finitely generated groups $G_n$ with $\depth(G_n) =\omega\cdot n$. Their groups are not
finitely presented when $n\ge 2$ and they ask \cite[Question 5.3]{BJ} if there might exist an alternative construction
that gives {\em finitely presented} groups $\G_n$ with $\depthRF(\G_n) =\omega\cdot n$. 
The first purpose of this note is to describe such a construction. Brody and Jankiewicz
also ask \cite[Question 5.4]{BJ}  if there might exist 
groups $M_n$ with $\depth(M_n) =\omega\cdot n+1$.   

\begin{thmA}\label{t1}
For every  positive integer $n$, there exists a finitely presented group $\G_n$ with $\depthRF(\G_n) =\omega\cdot n$ and
a finitely presented group $M_n$ with $\depthRF(M_n) =\omega\cdot n +1$.
\end{thmA}

The groups in Theorem \ref{t1} will be constructed in an inductive manner as amalgamated free products; the
proof relies on Bass-Serre theory.  
The construction is quite flexible, allowing one to impose extra conditions on the groups
constructed; for example, one can require the $\G_n$ to have finite classifying spaces. But  in an effort
 to make the ideas as clear as possible, I shall concentrate
on specific families  of examples, leaving the reader to consider variations. 

A third question of Brody and Jankiewicz  \cite[Question 5.5]{BJ}  concerns the existence of groups whose
depths are  $\omega^n$ and other ordinals.
They also ask about the difference between finitely generated and finitely presented groups in this context.

The following theorem gives a complete classification of the ordinals that arise as the residual finiteness
depths of finitely generated groups. We shall see in Section \ref{s:3} that the answer is the same
for countable groups.

\begin{thmA}\label{t3} 
Let $\alpha$ be an ordinal. There exists  a finitely generated group $J_\alpha$ for which $\depthRF(J_\alpha) =\alpha$
if and only if $\alpha$ is $0,\ 1$, a countable  limit ordinal, or the successor of a countable limit ordinal.  
\end{thmA}

There are uncountably many countable limit ordinals but only countably many finitely presented groups, so we obtain the following
consequence.

\begin{corA}
There are uncountably many ordinals that arise as the residual finiteness depth of a finitely generated group but not as
the  residual finiteness depth of a finitely presented group. 
\end{corA}
 
For ordinals $\alpha\ge \omega^2$, none of 
the groups constructed in the proof of Theorem \ref{t3} are  finitely presented, and for the moment 
I do not know how to realise any of these ordinals   as
the residual finiteness depth of a finitely presented group.

This paper is organised as follows. In Section \ref{s:1} we gather basic definitions, relate depth to cores
for countable ordinals, and prove some elementary lemmas describing the interplay of  these notions with Bass-Serre theory.
In Section \ref{s:2} we prove that every limit ordinal less than $\omega^2$ arises
as the residual finiteness depth of a finitely presented group.  In Section \ref{s:3} we prove that every countable limit ordinal
arises as the residual finiteness depth 
of a countable group, and deduce some related facts concerning successor ordinals; the 
constructions here involve a combination of wreath products and free products. In Section \ref{s:4}
we embed the countable groups of the previous section into finitely generated groups in a careful manner, enabling
us to show that every countable limit ordinal is the residual finiteness depth of a finitely generated group.
In Section \ref{s:5} we complete the proof of Theorem \ref{t1} by dealing with successor ordinals and in Section \ref{s:6}
we complete the proof of Theorem \ref{t3} by combining ideas from previous sections. 

\section{$\depth$ and residual cores}\label{s:1}

Brody and Jankiewicz \cite{BJ} introduce and explore extensions of the following  notion of {\em residual finiteness depth};
for our purposes this simplified definition will suffice. 
A similar notion is implicit in \cite{baum}.
 
\begin{definition}\label{defn:alpha rf}
Let $\alpha$ be an ordinal. A group $G$ is  \emph{$\alpha$-residually finite} if there exist 
subgroups $\{G_\beta\}_{\beta\leq\alpha}$ of  $G$ so that 
    \begin{itemize}
\item[•] $G_0= G$ and $G_\alpha=\{1\}$,
\item[•] $G_\gamma \le  G_\beta$ if $\gamma > \beta$,
\item[•] $[G_{\beta} \colon G_{\beta+1} ]$ is finite for all $\beta<\alpha$,
\item[•] $G_\lambda =\bigcap_{\beta<\lambda} G_\beta$ for all limit ordinals $\lambda \leq \alpha$.
    \end{itemize}
If $\alpha$ is the least ordinal such that $G$ is $\alpha$-residually finite, then $\depth(G):=\alpha$.
\end{definition}
 
Note that there are finitely generated groups which are not $\alpha$-residually finite for any $\alpha$, for example 
groups that have no non-trivial finite quotients. Note too that if $\depth(G)$ is not $0$ (the trivial group) or $1$ (finite groups), then it is either a limit ordinal or the successor of a limit ordinal: if $G$ is $(\alpha + n)$-residually finite, with $n$ finite, 
then in a filtration $\{G_\beta\}$ witnessing this, one can replace  $G_\alpha \ge G_{\alpha + 1} \ge \dots \ge G_{\alpha+n}$
by $G_\alpha \ge G_{\alpha+n}$ to see that  $\depth(G)\le \alpha + 1$. 

In the introduction we defined $\depth(G)$  in terms of cores, and it will be convenient to use the language of cores again in our proofs, so we make the following definition.

\begin{definition}\label{defn:core} Let  $G$ be a group. The  {\em residual  core} $\core(G)$ is
 the normal subgroup consisting of those elements of $G$ that have trivial image in every finite quotient of $G$.
Starting with $\core^0(G)=G$ and $\core^1(G)=\core(G)$, for successor ordinals (e.g. natural numbers) we define $\core^{\alpha + 1}(G):=\core(\core^\alpha(G))$, 
 and for limit ordinals 
$\core^\alpha(G) := \bigcap\{\core^\beta(G)\mid \beta<\alpha\}$.
\end{definition}

\begin{remarks}\label{interpolate}$\ $

(1) If $H$ is a countable group, then one can construct a sequence of finite-index subgroups
$H=H_0>H_1>\dots$ with $\bigcap_i H_i = \core(H)$ by choosing coset representatives $\{c_0=1, c_1, c_2,\dots\}$
for $\core(H)$ in $H$ and defining $H_i$ to be the intersection of $H_{i-1}$ with
the kernel of a map from $H$ to a finite group that maps $c_i$ non-trivially. 

(2) If $G$ is countable and $\depth(G)=\alpha$ then $\alpha$ is countable: if  $\alpha$ were not countable,
then by choosing an element $g_\beta\in \core^\beta(G)\ssm  \core^{\beta+1}(G)$ for each countable ordinal
$\beta$ we would obtain an uncountable subset of $G$.
\end{remarks}

The reader will recall that all limit ordinals have the form $\omega\cdot\beta$. (Throughout, it is assumed that 
the reader is familiar with elementary facts about ordinals and their arithmetic, as found in Chapter 8 of \cite{enderton} for example.)

\begin{lemma}\label{l:cores-to-depth} Let $G$ be a countable group and let $\alpha$ and $\beta$ be ordinals.
	\begin{enumerate} 
		\item  $\core^\beta(G)=1$ if and only if $G$ is $(\omega\cdot \beta)$-residually finite.
		\item  When $\core^\beta(G)$ is non-trivial, it is finite if and only if $G$ is $(\omega\cdot \beta+ 1)$-residually finite. 
		\item $\depth(G)=\omega\cdot \alpha+1$ if and only if $\core^\alpha(G)$ is finite but non-trivial. 
		\item  $\depth(G)=\omega\cdot \alpha$ if and only if 
		$\core^\alpha(G)=1$ and $\alpha$ is the least ordinal such that $\core^\alpha(G)$ is finite. 
	\end{enumerate} 
\end{lemma}

\begin{proof} For (1), 
assuming $\core^\beta(G)$ to be trivial, 
we must define a family of subgroups $\{G_\gamma\}_{\gamma\le\omega\cdot\beta}$ showing that
$G$ is $\omega\cdot\beta$-residually finite. First, we use Remark \ref{interpolate} to construct a sequence of finite-index subgroups $H_i(0)$ in $G$ that intersect in $\core(G)$. Next, for limit ordinals $\omega\cdot\delta \le \omega\cdot\beta$, 
we define  $G_{\omega\cdot\delta} = \core^\delta(G)$. To finish, for successor ordinals we
use Remark \ref{interpolate} again to construct a sequence of finite-index subgroups $H_i(\delta)$ in $\core^\delta(G)$
that intersect in $\core^{\delta+1}(G)$ and define $G_{\omega\cdot\delta + i} = H_i(\delta)$. 

For the reverse implication in (1), we take
$\{G_\gamma\}_{\gamma\le\omega\cdot\beta}$ to be a set of subgroups witnessing the fact that 
$G$ is $\omega\cdot\beta$-residually finite. Then, since  $\core(K)\le \core(H)$ for all groups $K<H$,  
an obvious transfinite induction shows that  $\core^\delta(G) \le G_{\omega\cdot\delta}$ for all $\delta\le\beta$. 

The passage from (1) to (2) is trivial. Items (3) and (4) follow immediately from (1) and (2).
\end{proof}
  
\subsection{Cores, depth, and graphs of groups}

I shall assume that the reader is familiar with the rudiments of Bass-Serre theory \cite{trees}.
If $\G $ is the fundamental group of a graph of groups $\GG$, the universal cover $\widetilde{\GG}$ (in the
category of graphs of groups) is the Bass-Serre tree for this splitting. Associated to each subgroup $H<\G$ one
has the graph of groups $\widetilde{\GG}/\!/H$ covering $\GG$. We shall be particularly interested in the case 
where the underlying graph $|\widetilde{\GG}/\!/H|$ is a tree, which it will be if and only if $H$ is generated by 
elliptic elements, i.e. by $\bigcup (H\cap G_v^\gamma)$, where the union is taken over all conjugates of the vertex
groups $G_v$  of $\GG$. Given a normal subgroup $K\ns\G$ we also consider the graph of groups $\GG_K$ 
(thought of as ``$\GG \mod K$") that is obtained
from $\GG$ without changing the underlying graph by replacing each of local groups $G$  with $G/(G\cap K)$,
taking the induced inclusions of edge groups. 
The kernel of the natural epimorphism $\G=\pi_1\GG\to\pi_1\GG_K$ is the subgroup of $K$ generated by elliptic elements.
The quotient map $\G\to \G/K$ factors through $\G\to \pi_1\GG_K$ giving
a natural epimorphism $\pi_1\GG_K\to \G/K$ that is inclusion on the local groups; this epimorphism is not injective in general.

\begin{lemma}\label{l:tree}
Let $\GG$ be a tree of groups with $\G= \pi_1\GG$.
If all of the edge groups of $\GG$ lie in $\core(\G)$, then $|\widetilde{\GG}/\!/\core(\G)|$ is a tree.
\end{lemma}

\begin{proof} 
We apply the preceding general considerations with $K=\core(\G)$.  The vertex groups of $\GG_K$ are residually finite, since the cores of the vertex groups of $\GG$ are contained in $K$, 
and the edge groups are trivial by hypothesis. Thus $\pi_1\GG_K$ is a free product of residually finite
groups. In particular it is residually finite, so the natural epimorphism $\pi_1\GG_K\to \G/\core(\G)$ splits
and $\core(\G) = \ker(\G\to\pi_1\GG_K)$, which is generated by elliptic elements.
\end{proof} 

\begin{remark}
When we interpolate between $\G$ and $\core(\G)$ with finite-index subgroups $\G=\G_0>\G_1>\dots$, the 
graphs $|\widetilde{\GG}/\!/\G_i|$ need not be trees. For example, in the infinite dihedral group $\G=(\Z/2)\ast(\Z/2)$, if $\G_1$
is the infinite cyclic subgroup of index two, then $|\widetilde{\GG}/\!/\G_1|$ has two vertices and two edges.
\end{remark}

\begin{lemma}\label{l:free-prod} If a countable group $\G$  is the free product of non-trivial groups $A_i$, then
\begin{enumerate}
\item $A_i\cap \core^\alpha(\G) = \core^\alpha(A_i)$ for any ordinal $\alpha$, and 
\item $\core^\alpha(\G)$ is a free product of copies of the groups $\core^\alpha(A_i)$.
\item If $\sup_i \depth(A_i)$  is a limit ordinal,  then $\depth(\G) = \sup_i \depth(A_i)$;
\item if $\sup_i \depth(A_i)$  is a successor ordinal,  $\depth(\G) = \sup_i \depth(A_i)+\omega.$ 
\end{enumerate} 
\end{lemma}

\begin{proof} Assertion (1) holds more generally for retracts $A$ of $\G$: 
 for $\alpha=1$ it is the obvious assertion that   
an element of $A$  survives in a finite quotient of $A$ if and only if it survives in some finite quotient of $\G$;
for arbitrary $\alpha$, the result follows by transfinite induction, noting repeatedly that a retraction $\G\to A$
restricts to a retraction $\core(\G)\to\core(A)$. 

For (2), we consider the action of $\core^\alpha(\G)$ on the Bass-Serre tree 
for the splitting $\GG$ of $\G$ as a free product of the $A_i$. 
The vertex groups in the quotient graph-of-groups are conjugates of  $A_i\cap \core^\alpha(\G)$,
which by (1) are copies of $\core^\alpha(A_i)$. Proceeding by transfinite induction,  with 
Lemma \ref{l:tree} as the base step, we shall argue that
the underlying graph $|\widetilde{\GG}/\!/\core^\alpha(\G)|$ is a tree. In the inductive
step, we assume that this is the case for ordinals $\beta<\alpha$. 

At successor ordinals $\alpha=\alpha_0+1$, we can apply Lemma \ref{l:tree} with $\core^{\alpha_0}(\G)$ in place of
$\G$, since  $\widetilde{\GG}/\!/\core^{\alpha_0}(\G)$ (which is a {\em{tree}} of groups, by induction) has the same
universal cover as $\GG$.
To clarify the argument at limit ordinals, we first note that 
an arbitrary normal subgroup $N$ in  $\G = \bigast_iA_i$ will be generated by its elliptic
elements $\bigcup_{i,\gamma}(N\cap A_i^\gamma)$ if and only if there does {\em not} exist
an  element $\nu\in N$ whose free-product normal form is $\nu =a_1\dots a_m$, 
with all $a_j\in A_{i(j)}\ssm N$. (Such an element $\nu$ determines a reduced loop in the quotient
graph of groups $\GG/\!/N$.) We apply this observation with $N=\core^\alpha (\G)$:
 if there were to exist $\nu\in\core^\alpha(\G)$ with $\nu =a_1\dots a_m$ and
all $a_j\in A_{i(j)}\ssm \core^\alpha(\G)$ then, because $ \core^\alpha(\G) = \bigcap_{\beta<\alpha} \core^\beta(\G)$,
this would imply that $a_j\in A_{i(j)}\ssm \core^\beta(\G)$ for all $a_j$ when $\beta<\alpha$ is 
sufficiently large; but this cannot happen, because it implies that $\core^\beta(\G)$ is not generated by its
elliptic elements, contrary to the inductive hypothesis. This completes the proof of (2).

\smallskip
Towards proving (3) and (4), we  consider $\sigma =  \sup_i \depth(A_i)$. 
If $\sigma$ is a limit ordinal $\sigma = \omega\cdot\alpha$ then  $\core^\alpha(A_i)=1$ for all $i$,
so from (2) we have  $\core^\alpha(\G)=1$, hence $\depth(\G)\le \omega\cdot\alpha = \sigma$. 
And
by applying to $A_i<\G$ the observation that  $H<K$ implies $\depth(H)\le \depth(K)$, we have $\sigma\le \depth(\G)$,
so   (3) is proved.
For (4), we suppose $\sigma =  \omega\cdot\beta + 1$,
which means that $\core^\beta(A_i)$ is finite for all indices $i$ but $\core^\beta(A_i)$ is non-trivial for some $i$.
In this case, $\core^\beta(\G)$, which has more than one free factor isomorphic to $\core^\beta(A_i)$ for each $i$,
is a free product of finite groups and is infinite. As any free product of finite groups is residually finite, we conclude
that $\core^\beta(\G)$ is infinite but $\core^{\beta+1}(\G)=1$, hence $\depth(\G) = \omega\cdot(\beta + 1) = \sigma +\omega$.
\end{proof}

\subsection{Direct Products}

For completeness, we record the analogue of Lemma \ref{l:free-prod} for direct sums. The proof in this case
is straightforward, so we omit it.

\begin{lemma}\label{l:dir-prod} If a countable group $\G$  is the direct sum of non-trivial groups $A_i$, then
\begin{enumerate}
	\item $A_i\cap \core^\alpha(\G) = \core^\alpha(A_i)$ for any ordinal $\alpha$, and
	\item $\core^\alpha(\G)$ is the direct sum of the groups $\core^\alpha(A_i)$.
	\item If $\sup_i \depth(A_i)$  is a limit ordinal,  then $\depth(\G) = \sup_i \depth(A_i)$;
	\item if $\sup_i \depth(A_i)$  is a successor ordinal,  then $\depth(\G) = \sup_i \depth(A_i)$ if there
are only finitely many $A_i$ that realise the supremum and  $\depth(\G) = \sup_i \depth(A_i)+\omega$ if there are  
infinitely many.
\end{enumerate} 
\end{lemma}

\subsection{Some more basic facts}

The following observations will be useful throughout and we shall often
use them without comment.  

\begin{lemma} For all ordinals $\alpha$ and groups $G$,
\begin{enumerate}
\item if $H\le G$ then $\core^\alpha(H)\le \core^\alpha(G)$; 
\item if $H\le \core^\beta(G)$ for all $\beta<\alpha$ then $\core(H)\le \core^\alpha(G)$;  
\item if $p:G\twoheadrightarrow Q$ is a surjection with kernel $K$, then $\core^\alpha(G)\le p^{-1}(\core^\alpha(Q))$
with equality if $K\le \core^\alpha(G)$. 
\end{enumerate}
\end{lemma}

\section{Finitely presented groups realising the ordinals $\omega \cdot n$}\label{s:2}

We want to prove that $\omega \cdot n$ is the residual finiteness depth of a finitely presented group, 
proceeding by induction on $n$.  In order to make the induction run smoothly, it is convenient to
make the inductively-constructed groups $\G_n$ have the following properties:
\begin{enumerate}
\item $\G_n$ is finitely presented;
\item $\depthRF(\G_n)= \omega\cdot n$;
\item $\G_n$ has a generating set $\{a_0,\dots,a_m\}$ for which there is an epimorphism $h_n:\G_n\to\Z$
with $h_n(a_i)=h_n(a_0)$ for $i=1,\dots,m$.  (The integer $m$ depends on $n$.)
\end{enumerate}

In the inductive step, we construct a group $\G_{n}^0$ satisfying (1) and (2) then define $\G_{n}=\G_{n}^0\ast\Z$,
invoking the following observation to get condition (3).

\begin{lemma}\label{l2} 
If $\G_n^0$ satisfies (1) and (2), then $\G_n=\G_n^0\ast\Z$ satisfies (1), (2) and (3).
\end{lemma}

\begin{proof}
If $\G_n^0$ is generated by $b_1,\dots,b_m$ and $\zeta$ generates $\Z$, define $a_0=\zeta$ and $a_i=b_i\zeta$. 
Lemma \ref{l:free-prod} 
assures us that  $\G_n\ast\Z$ satisfies condition  (2), and killing $\G_n^0$ yields $h_n$.
\end{proof}

\begin{remark}\label{r:BE}
The key point about property (3) is that it ensures that all of the cyclic subgroups $\<a_i\> <\G_n$ are retracts of $\G_n$.
To see why this might be useful, recall  from \cite{BE}
that if $\GG$ is a tree of groups where each vertex group is residually finite and 
each edge group is a retract of the vertex groups into which it includes, then $\pi_1\GG$ is residually finite.
\end{remark}

It will be convenient to have a finitely presented group $\La$ for which $\coreRF(\La)\cong\Z$. A famous example
of such a group is Deligne's central extension \cite{deligne}
of ${\textrm{Sp}}(2g,\Z)$, any $g\ge 2$,  for which $\coreRF(\La)$ is the subgroup
of index 2 in the kernel of the central extension
\begin{equation}\label{e1}
1\to  \Z \to \La\to {\textrm{Sp}}(2g,\Z)\to 1.
\end{equation}
Henceforth, $\La$ will always denote this central extension.  Later on, we shall also need  a central extension 
of a residually finite group that has finite centre and
is not residually finite. Such a group can be obtained from  $\La$  by
killing a subgroup of odd  index $d>2$ in the kernel of the extension (\ref{e1}),
\begin{equation}\label{e:d}
1\to \Z/d\Z \to \DL \to {\textrm{Sp}}(2g,\Z)\to 1.
\end{equation}  
The value of $d>2$ will not matter, but we fix $g\ge 3$ to ensure that $\DL$ is perfect (which
follows immediately from the fact that ${\textrm{Sp}}(2g,\Z)$ is perfect,  since the central $\Z/d\Z$
dies in every finite quotient of $\DL$).

\begin{theorem}\label{t:2.2}
For every  positive integer $n$, there exists a finitely presented group $\G_n$ with $\depthRF(\G_n) =\omega\cdot n$.
\end{theorem} 

\begin{proof}
We argue by induction on $n$. For $\G_1$ we can take $\Z$, for example.
For $\G_2$ we can take $\G_2^0\ast\Z$ where $\G_2^0\cong\La$ is the group from (\ref{e1}).  In the inductive step, we
assume  $n\ge 2$ and that $\G_n$ satisfies the conditions (1) to (3) described above, and we must construct $\G_{n+1}^0$
satisfying (1) and (2).
Define $\G_{n+1}^0$ to be the
group obtained by amalgamating $m+1$ copies of $\La$ with $\G_n$, where the $i$-th copy $\La_i$  of  $\La$
is amalgamated by identifying a generator $\zeta_i$ of $\coreRF(\La_i)\cong\Z$ with the generator $a_i\in\G_n$.

In any finite image of $\G_{n+1}^0$, each $\zeta_i$ must ``die" (i.e. map trivially), so $\G_n$ must
die, since each of its generators has been identified with some $\zeta_i$. Thus
every map from $\G_{n+1}^0$ to a finite group will factor through the epimorphism 
$\rho: \G_{n+1}^0\to \bigast_{i=0}^m\La_i/\<\zeta_i\>$. This free product of copies of $\La/\<\zeta\>$
 is residually finite, so the kernel of $\rho$ is
$\coreRF(\G_{n+1}^0)$.  Bass-Serre theory, as in the proof of Lemma \ref{l:free-prod}, tells us that this
kernel is the fundamental group of a graph of groups $\T$ some of whose vertex groups are conjugates of $\G_n$
while the others (from the $\La_i$ vertices) are infinite cyclic groups; each of the latter type is generated by 
each of the edge groups incident at that vertex, since each is a conjugate of  $\<\zeta_i\>$
and $\<\zeta_i\>$ is normal in $\La_i$. Moreover, Lemma \ref{l:tree} tells us that the underlying graph of $\T$ is a tree.

The map $h_n$ in condition (3) provides a map onto $\Z$ from each
vertex group of type $\G_n$ in the decomposition $\mathcal{T}$,
and these  maps agree on intersections (the edge groups). Thus we obtain a map $h_{n+1}$ from $\pi_1\T=\core(\G_{n+1}^0)$
onto $\Z$ that restricts to an isomorphism on each edge  group and on each of the vertex groups that is 
cyclic. It follows that the subgroup of $\core(\G_{n+1}^0)$ consisting of elements that die in every finite quotient
of $\coreRF(\G_{n+1}^0)$
intersects the edge groups and cyclic vertex groups trivially, and therefore this subgroup
(which is $\core^2(\G_{n+1}^0)$)  is the
fundamental group of a graph of groups whose edge groups are trivial and whose vertex groups are either
trivial or else conjugates in $\core(\G_{n+1}^0)$ of $I_n:=\G_n\cap \core^2(\G_{n+1}^0)$, which we claim is $\core(\G_n)$. 

It is obvious that $I_n$ contains  $\coreRF(\G_n)$, so to prove the claim it is enough to argue that each
$\g\in \G_n\ssm \coreRF(\G_n)$ survives in some finite quotient of $\core(\G_{n+1}^0)$. By definition, 
there is a finite quotient $p_\g:\G_n\to Q$ with $p_\g(\g)\neq 1$, and we will be done if we can extend $p_\g$
to the whole of $\coreRF(\G_{n+1}^0)$. For this, we consider the components of the graph-of-groups obtained
when we delete the base vertex (where the local group is $\G_n$) from 
the tree of groups $\T$ described above. There is one such component $C_i$ of $\T$ 
  for each of the generators $a_0,\dots,a_m$ of $\G_n$, and $\pi_1C_i<\pi_1\T  = \coreRF(\G_{n+1}^0)$
  intersects $\G_n$ in $\<a_i\>$.
If $p_\g(a_i)$ has order $e_i$,  we extend $p_\g$ to $\pi_1C_i$ by composing the surjection
$h_{n+1}: \pi_1C_i\to\Z$ (defined above) 
with $\Z\to \Z/e_i\Z$. This completes the proof of the claim. 

At this point, we have proved  that  $\core^2(\G_{n+1}^0)$  is a free product
of   (infinitely many) copies of $I_n=\coreRF(\G_n)$ and possibly a free group. By induction,  
$\depthRF(\G_n)=\omega\cdot n$, so
Lemma \ref{l:free-prod} tells us that $\depthRF(\core^2(\G_{n+1}^0))=\omega\cdot (n-1)$. Thus
$\depthRF(\G_{n+1}^0)=\omega\cdot (n+1)$. This completes the induction.  
\end{proof}

Later, we shall need the following additional property of our construction. 

\begin{addendum} \label{ad1}
For $n\ge 3$, the subgroup
$\core^{n-1}(\G_n)$ is a free product of a free group and conjugates of $\core(\G_2)$.
\end{addendum}

\begin{proof} In the last paragraph of the preceding proof, we saw that $\core^2(\G_{n+1}^0)$ was
a free product of copies of $\coreRF(\G_n)$ and a free group. The addendum follows by a simple induction.
For  the inductive step,  since  $\G_{n+1}=\G_{n+1}^0\ast\Z$, we use Lemma \ref{l:free-prod} to see
that   $\core(\G_{n+1})$ is  the free product of copies of $\core(\G_{n+1}^0)$,    hence
$\core^2(\G_{n+1})$ is a free product of copies of $\core^2(\G_{n+1}^0)$, which 
is a free product of copies of $\coreRF(\G_n)$ and a free group.  Repeated application of 
Lemma \ref{l:free-prod} then tells us that $\core^{n}(\G_{n+1})$ is a free product of 
copies of $\core^{n-1}(\G_n)$.
\end{proof}
 
\begin{remark}\label{r:better2} 
In the proof of Theorem \ref{t:2.2}, it was convenient to name a specific $\G_1$ and $\G_2$ before 
starting the inductive step, but we could have made many other choices for these groups. 
If we chose as $\G_1$ any   residually finite group with generators $\{a_0,\dots,a_m\}$ as in condition (3),
then we could get $\G_2^0$ by amalgamating $m+1$ copies of $\Lambda$ with $\G_1$, as in the inductive step of 
the proof,  identifying the generator $\zeta_i$ of the core of the $i$-th copy  $\La_i$ 
with $a_i\in\G_1$. 
Arguing as in the second paragraph of the proof, we would then get $\core(\G_2)$ to be the fundamental group of a 
tree of groups 
with cyclic edge groups and vertex groups that are either cyclic or copies of $\G_1$.  The  
 inclusions of edge groups in this tree of groups are either
isomorphisms or else inclusions into copies of $\G_1$ of one of the cyclic groups $\<a_i\>$.
In particular, all of the edge groups are retracts in the adjacent vertex groups, so the
fundamental group $\core(\G_2)$ is residually finite,  by \cite{BE}. Hence $\depth(\G_2) = \omega\cdot 2$,
as required.  The remainder of the proof then continues as before.
\end{remark}   

\section{Residual finiteness depth for countable groups}\label{s:3}

The purpose of this section is to prove the following result, a step towards  Theorem \ref{t3}.

\begin{proposition}\label{t:cbl} For every countable limit ordinal $\alpha$, 
there exists  a countable group $G$ with $\depthRF(G) =\alpha$. 
\end{proposition}

\subsection{Wreath Products}

The  results in this subsection can be gleaned from \cite{BJ}, but we present them in a way
that is  convenient to our aims.
We use the standard notation $A\wr B$ for the restricted wreath product, i.e. the semidirect product
$P_{A,B}\rtimes B$ where $P_{A,B}=\oplus_{b\in B} A_b$ is the direct sum of copies of $A$ indexed by $B$, with
$B$ acting by left-multiplication on the index set. We also need to consider more general {\em{permutational
wreath products}} $A\wr_{I} C = P_{A,I}\rtimes C$, 
where $P_{A,I}$ is the direct sum of copies of $A$ indexed by a set $I$
and $C$ acts by a homomorphism $C\to {\textrm{Perm}}(I)$.  We shall only be interested in the setting
where the action of $C$ on $I$ is free, which is case 
when $C$ is a subgroup of $B$ and the action is left multiplication on $B=I$; in this case,
$A\wr_{B} C$ is the subgroup of $A\wr B$ generated by $P_{A,B}$ and $C$.

\begin{lemma}\label{l:wr} Let $A\wr_{I} C=P_{A,I}\rtimes C$
 be a permutational wreath product where the action of $C$ on $I$ is free.
If $A$ is perfect and $C$ is infinite, then $\core(A\wr_{I} C)=P_{A,I}\rtimes \core(C)$.
\end{lemma}

\begin{proof}
For all $i\in I$ and $x,y\in A_i$ and any finite quotient $p: A\wr_{I} C\to Q$
we have $p(c)=1$ for some non-trivial $c\in C$. Writing $g^c$ for conjugation of $g$ by $c$,
\[
p([x,y]) = [p(x), p(y)] = [p(x), p(y)^{p(c)}] = [p(x), p(y^c)] = p([x, y^c]) =p(1)=1,
\]
where the penultimate equality comes from the fact that $y^c\in A_{c\cdot i}$ commutes with $x\in A_i$
since $c\cdot i\neq i$. Thus $[A_i, A_i]\subseteq \core(A\wr_{I} C)$ for all $i$. But $A_i$ is perfect,
so $A_i=[A_i,A_i]$ and $P_{A,I}\subseteq \core(A\wr_{I} C)$. 

Conversely, if $g\notin P_{A,I}\rtimes \core(C)$, then $g$ projects to $C\smallsetminus \core(C)$ and therefore 
survives in some finite quotient of $C$ and hence of $A\wr_{I} C$.
\end{proof} 

Arguing by transfinite induction, we deduce:

\begin{corollary}\label{c:wr} For all ordinals $\alpha$, if $A$ is perfect and $\core^\beta(B)$ is infinite
for all $\beta<\alpha$, then $\core^{\alpha}(A\wr B)=P_{A,B}\rtimes \core^\alpha(B)$.  
\end{corollary}
 
\begin{proposition}\label{p:wr}
Let $A\neq 1$ be a countable group that is  perfect and let $B$ be a countable group that is infinite
and $\beta$-residually finite for some $\beta$.
\begin{enumerate}
\item If $\depth(A)$ is a limit ordinal, then 
\[\depth(A \wr B) = \depth(B) + \depth(A).\]
\item If $\depth(A)$ is a successor ordinal, then 
\[\depth(A\wr B) = \depth(B) + \depth(A) + \omega.\]
\end{enumerate}
(Note that in both cases, $\depth(A\wr B)$ is a limit ordinal.)
\end{proposition}

\begin{proof}
If $\depth(B)=\omega\cdot\alpha$ or $\omega\cdot\alpha +1$, then $\core^\alpha(B)$ is finite, so $\core^\alpha(A\wr B)$ 
contains $P_{A,B}$ as a subgroup of finite index, by Corollary \ref{c:wr}. Passing to a subgroup of
finite index does not affect the residual finiteness depth
 of an infinite group, so  $\depth(A\wr B) = \depth(B) + \depth(P_{A,B})$.

Lemma \ref{l:dir-prod} tells us that  $\depth(P_{A,B})=\depth(A)$ if $\depth(A)$ is a limit ordinal 
and   $\depth(P_{A,B})=\depth(A)+\omega$ otherwise. 
\end{proof} 

Brody and Jankiewicz \cite{BJ} proved that there are finitely generated groups with residual finiteness depth $\omega\cdot n$
for any $n\in\N$ by repeatedly applying the following consequence of Proposition \ref{p:wr}.

\begin{corollary}
If $\alpha$ is the residual finiteness depth of a finitely generated (resp.~countable) group, then so is $\alpha + \omega$.
\end{corollary}

\begin{proof} The assertion is trivial if $\alpha \in\{0,1\}$, so assume $\alpha$ is infinite.
In Proposition \ref{p:wr}, take $A\neq 1$ to be a finitely generated, perfect, residually finite
 group and choose $B$ with $\depth(B)=\alpha$.
\end{proof}

\subsection{Proof of Proposition \ref{t:cbl}} The trivial group has residual finiteness depth $0$,
for non-trivial finite groups $\depth(G)=1$, and for residually-finite  groups that are infinite $\depth(G)=\omega$.
For other limit ordinals $\alpha$, we proceed by transfinite induction, assuming that for each limit ordinal $\beta<\alpha$
there exists a group $G_\beta$ with $\depth(G_\beta)=\beta$. If $\alpha = \alpha_0 + \omega$, then 
we take a non-trivial finite perfect group $A$, define $G_{\alpha}= A\wr G_{\alpha_0}$ and appeal to Corollary \ref{c:wr}.
If $\alpha$ is not of this form, then it is the supremum of the  (countably many) limit ordinals $\beta<\alpha$ and we define
$G_{\alpha}$ to be either the free product or the direct sum of the inductively defined groups $G_{\beta}$ corresponding to these ordinals.
Lemma \ref{l:free-prod}(3) and Lemma \ref{l:dir-prod}(3)  assure us that $\depth(G_{\alpha})=\alpha$ in both cases. \qed

\subsection{Attaining some successor ordinals}

\begin{proposition}\label{p:omega+1}
If $\alpha$ is the residual finiteness depth of a finitely generated (resp.~countable) group, then so is $\alpha + \omega+1$.
\end{proposition}

\begin{proof} Let $B$ be a finitely generated group with $\depth(B)=\alpha$. (The countable case is entirely
similar.)
Let $A$ be a finitely generated perfect group that has a finite central subgroup  $Z=\core(A)\neq 1$
-- one of Deligne's group $\DL$ from (\ref{e:d}) will do nicely (with $d$ odd and $g\ge 3$). Let $\-{A}=A/Z$ and
consider the central extension $E$ of $\overline{A}\wr B$ that is obtained from $A\wr B$ by imposing relations
that identify the copy of $Z$ in each direct factor of $P_{A,B}$ with a single copy of $Z$. In other words, $E$
is the quotient of $A\wr B$ by the relations $(b^{-1}zb = z \mid \forall b\in B,\, z\in Z)$. 
We have the central extension
\[
1\to Z \to E \to \overline{A}\wr B\to 1
\]
and the preimage in $E$ of each summand $\overline{A}_b$ in $P_{\overline{A},B}$ is a copy of $A$.

If $\alpha=\omega\cdot\gamma$ then as in Corollary \ref{c:wr}, we see that $\core^{\gamma}(E)$ is the preimage  
of $P_{\overline{A},B}< \overline{A}\wr B$.  This preimage is the quotient of $P_{A,B}$
obtained by identifying the central  $Z_b<A_b$ of the summands to a single $Z$. Thus, since
$Z_b =\core(A_b)$, the core of $\core^{\gamma}(E)$ is $Z$. Therefore 
\[
\depth(E) =  \alpha + \depth(\core^{\gamma}(E)) =  \alpha + \omega + 1.
\]
If $\alpha=\omega\cdot\gamma +1$, then we apply the preceding argument to $\omega\cdot\gamma$, noting that
$ \alpha + \omega + 1 = \omega\cdot\gamma + \omega +1$.
\end{proof}

\begin{corollary}
If $\alpha$  is the residual finiteness depth of a finitely generated (resp.~countable) group, then  so
is $\alpha + \omega\cdot n + 1$, for all positive $n\in\mathbb{N}$.
\end{corollary}
 
\begin{proof}
Apply Proposition \ref{p:omega+1} repeatedly, noting that 
$(\alpha + \omega\cdot n+1) + \omega +1 = \alpha + \omega\cdot (n+1) +1$.
\end{proof}

The central idea of the  proof of Proposition \ref{p:omega+1} applies more generally:
instead of looking for $A$ and $Z=\core(A)$ with $A/Z$  residually finite, we could look for $A$ and $Z$ 
with $\depth(A/Z) = \beta$ and $Z=\core^\beta(A)$. Then, arguing in the same manner, we obtain:

\begin{proposition}\label{p:beta+1} 
If $\alpha$ is the residual finiteness depth of a finitely generated (resp.~countable) group 
and $\beta+1$ is the residual finiteness depth of a finitely generated (resp.~countable)  group $A$ 
with $\core^{\beta}(A) < A$ central, then there is a finitely generated (resp.~countable) group
with residual finiteness depth $\alpha + \beta+1$.
\end{proposition}

By applying the proposition to the groups from Theorem \ref{t1}, we obtain the following result, which will
be subsumed into Theorem \ref{t3}.

\begin{corollary}
If $\alpha$  is the residual finiteness depth of a finitely generated (resp.~countable) group, then  so
is $\alpha + \omega\cdot n + 1$, for all positive $n\in\mathbb{N}$.
\end{corollary}

\section{Countable limit ordinals as residual finiteness depths}\label{s:4}

The purpose of this section is to prove the following result.

\begin{theorem}\label{t:fg-limit} For every countable limit ordinal $\alpha$, 
there exists  a finitely generated group $\G$ with $\depthRF(\G) =\alpha$.
\end{theorem}

Our proof divides into the case $\alpha < \omega^2$, which is covered by Theorem \ref{t:2.2},
and the case $\alpha \ge \omega^2$. A key technical point in the second case is that
if $\depth(G)=\alpha$ then $\depth(G) = \depth (\core(G))$, because
$\omega + \alpha = \alpha$.

\subsection{Limit ordinals $\alpha\ge\omega^2$} In this regime, Theorem \ref{t:fg-limit}
is an immediate consequence of Proposition \ref{t:cbl} and the following controlled-embedding lemma. 

\begin{proposition}\label{p:3gen} Let $\G$  be a countable group and suppose that
$\depth(\G)\ge \omega^2$ is a limit ordinal. Then, there exists a 3-generator group $\G^+$ and
an embedding $\G\hookrightarrow\G^+$ such that   $\depth(\G) = \depth(\G^+)$.
\end{proposition}
  
At several points in the following proof we shall use the obvious fact that 
$H<K $ implies $\depth(H)\le \depth(K)$.

\begin{proof} 
We list the elements of $\G$ without repetition, $\G=\{g_0=1,g_1,g_2,\dots\}$. 
Let $\G_0=\G\ast\Z$, fix a generator $z$ for $\Z$,
and for each $i\in\N$ define $a_i=zg_i$.  Lemma \ref{l:free-prod} assures us that $\depth(\G) = \depth(\G_0)$.
Note that each $a_i$ has infinite order in $\G_0$. Let $\G_1$ be the HNN extension
of $\G_0$ with infinitely many stable letters $t_i \ (i\ge 1)$, where $t_i$ conjugates $a_i$ to $a_0$:
\[
\G_1 = (\G_0,\, t_1,t_2,\dots \mid t_i^{-1}a_it_i=a_0\ \text{for }i\ge 1).
\]
Note that $\G_1$ is generated by $\{a_0, t_1, t_2,\dots \}$ and that the $t_i$ freely generate a free group 
$T$ onto which $\G_1$ retracts with kernel $\<\!\<a_0\>\!\>$.  We claim that $\core(\G_1)$
is the free product of a free group and copies of $\G_0\cap \core(\G_1)$. Then, as the depth of 
$\G_0\cap \core(\G_1)$ is at most $\depth(\G_0)$, using Lemma \ref{l:free-prod} we have   
\begin{equation*}
\begin{split} 
\depth(\G_1) & = \omega + \depth(\core(\G_1))\\
& \le  \omega + \depth(\G_0) = \omega + \depth(\G) = \depth(\G),
\end{split}
\end{equation*}
where the last equality holds because $\depth(\G) \ge \omega^2$.  As $\G<\G_1$,
we conclude that  $\depth(\G_1) = \depth(\G)$.

To prove the claim, consider the graph-of-groups decomposition $\GG$ of $\G_1$ corresponding to its HNN structure. Each
edge group is a conjugate of $\<a_0\>$, so it  intersects $\core(\G_1)$ trivially because killing all of 
the stable letters $t_i$ retracts $\G_1$  onto $\<a_0\>$. Therefore, the covering of $\GG$ with fundamental group
$\core(\G_1)$ has trivial edge groups and has vertex groups that are conjugates of $\G_0\cap \core(\G_1)$. 

For the next step in the construction,
we choose a free group of infinite rank $L$ in $F_2 = {\textrm{Free}}(b,c)$ and identify this
with $T$ to form the amalgamated free product
\[
 \G^+ := \G_1\ast_{T=L} F_2.
 \]
Note that $\G^+$ is generated by $\{a_0, b, c\}$. The retraction of $\G_1$ onto $T$ extends to a retraction of $\G^+$
onto $F_2$, which is residually finite, so $\core(\G^+)$ intersects each conjugate of $F_2$ trivially, and hence is the
fundamental group of a graph of groups with trivial edge groups and vertex groups that are either trivial or conjugates
of $\core(\G^+)\cap \G_1$. Thus $\core(\G^+)$ is a free product of a free group and copies of $\core(\G^+)\cap \G_1$.
As above, this implies    
\[
\depth(\G^+) = \omega +  \depth(\core(\G^+)) \le  \omega + \depth(\G_1) = \depth(\G_1).
\]
Hence $\depth(\G^+) = \depth(\G_1) = \depth(\G)$.
\end{proof}

\section{Finitely presented groups realising the ordinals $\omega \cdot n +1$}\label{s:5}
In  \cite[Question 5.4]{BJ}, Brody and Jankiewicz  ask if there exist finitely presented groups $M_n$ with $\depthRF(M_n) =\omega\cdot n +1$. In this section we will see that the proof of Theorem \ref{t1}
can be modified to arrange this. 
For this construction, we need a central extension 
of a residually finite group that has finite centre and
is not residually finite.  We will use the group $\DL$  from (\ref{e:d}) with $d$ odd and $g\ge 3$ fixed.

\begin{theorem}\label{t3+1} 
For every positive integer $n$ and odd integer $d\ge 3$,  there exists a finitely presented group $M_n$
with centre $Z(M_n)\cong\Z/d\Z$  such that $\depth(M_n) =\omega\cdot n +1$.
\end{theorem} 

\begin{proof} 
We follow the construction of $\G_n$ in Theorem \ref{t:2.2} with a more careful choice of $\G_1$ and $\G_2$.
Specifically,  we
take $\G_1^0 = {\textrm{Sp}}(2g,\Z)$,  for some fixed $g\ge 2$, then $\G_1=\G_1^0\ast\Z$; then we build $\G_2$ as described in 
Remark \ref{r:better2}. 
In the inductive step we define $\G_{n+1}$ with $\depth(\G_{n+1})=\omega\cdot (n+1)$ by 
amalgamating $\G_{n}$ with $m+1$ copies of $\La$, as in the proof of Theorem \ref{t:2.2}, to obtain $\G_{n+1}^0$
then take $\G_{n+1}=\G_{n+1}^0\ast\Z$. This family of groups comes with embeddings $\G_n\hookrightarrow\G_{n+1}$.
Addendum \ref{ad1} tells us that $\core^{n-1}(\G_n)$ is a free product of a free group and conjugates of $\core(\G_2)$,
and in Remark \ref{r:better2} we saw that $\core(\G_2)$ is a residually finite group that contains conjugates of
$\G_1$. Thus $\core^{n-1}(\G_n)$ is a residually finite group that contains conjugates of $\G_1$.

With this structure in hand, we will be done if we can describe a central  extension 
\[
1\to \Z/d\Z \to M_n \overset{p}\to \G_n\to 1
\]
that restricts over ${\textrm{Sp}}(2g,\Z)= \G_1^0 < \G_n$ to Deligne's extension  (\ref{e:d}), because this will tell us that
$\core^{n-1}(M_n)$ is a central extension
\[
1\to \Z/d\Z \to \core^{n-1}(M_n) \overset{p}\to \core^{n-1}(\G_n)\to 1
\]
with $\core^{n-1}(\G_n)$ residually finite and $\Z/d\Z = \core(p^{-1}(\G_1^0)) = \core^{n}(M_n)$. 

To this end, we define $M_1$ by amalgamating the group ${\DL}$
from  (\ref{e:d}) with $(\Z/d)\times\Z$, identifying the central subgroup $\Z/d\Z<\DL$ with
$(\Z/d)\times\{1\}$.  
Then, 
proceeding by induction, assuming that the desired central extension $M_{n}$ has been constructed 
over $\G_{n}$, 
we define $M_{n+1}^0$ to be the central extension of $\G_{n+1}^0$ that restricts to $M_{n}$
over $\G_{n}$ and restricts to the trivial extension $(\Z/d\Z)\times \La$ over each of the $m+1$ copies of $\La$
attached to $\G_{n}$ in the proof of Theorem \ref{t:2.2};
we then identify the centre of $M_{n+1}^0$ with $\Z/d\times \{1\}$ to
amalgamate $M_{n+1}^0$ with $(\Z/d\Z)\times\Z$.
More explicitly, to obtain $M_{n+1}^0$
we  fix a generator $\zeta$ for the central $\Z/d<M_{n}$, then amalgamate $M_n$ with $m+1$ copies of $(\Z/d\Z)\times\La$,
along copies of $(\Z/d\Z)\times\Z$
identifying $\zeta$ with a generator of $(\Z/d\Z)\times\{1\}$ and $\tilde{a}_i$ with the generator of $\core(\La)$,
where $\{\tilde{a}_0,\dots,\tilde{a}_m\}$ are preimages in $M_{n}$ of the generators $a_i\in \G_{n}$ used in
the proof of Theorem \ref{t:2.2}. (An important
point to observe here is that  because  $\<a_i\><\G_{n}$ is infinite cyclic,
its preimage $\<\zeta,\, \tilde{a_i}\>$ in $M_{n}$ is $(\Z/d\Z)\times\Z$.) 
Finally,  we obtain $M_{n+1}$ by amalgamating $M_{n+1}^0$ with $(\Z/d\Z)\times\Z$, again identifying $\zeta$ with
a generator of $(\Z/d\Z)\times\{1\}$. 
By construction, $\<\zeta\>\cong\Z/d\Z$ is central in $M_{n+1}$ with quotient $\G_{n+1}=\G_{n+1}^0\ast\Z$, which is
centreless.
\end{proof}

\section{Successors of limit ordinals}\label{s:6}

\def\wt{\widetilde}
\def\Gt{\widetilde{J}}
\def\Gta{\widetilde{J}_\alpha}

The use of the prime $3$ in the following statement is arbitrary. 

\begin{theorem}\label{t:+1}
For every countable limit ordinal $\alpha=\omega\cdot\gamma$ there exists a finitely generated group $\Gta$ with
$\depth(\Gta) = \alpha +1$ and $\core^\gamma(\Gta)=Z(\Gta)\cong\Z/3$.
\end{theorem}

\begin{proof}
We again proceed by transfinite induction.  
The limit ordinals $\alpha <\omega^2$ have the form $\omega\cdot n$, 
and for these we can take the groups $M_n$ constructed in Theorem \ref{t3+1}.

Thus we may assume that $\alpha\ge \omega^2$ and we may assume that $\Gt_\beta$ has
been constructed for limit ordinals $\beta < \alpha$. We shall again exploit the fact that
$\alpha\ge \omega^2$ implies $\omega + \alpha=\alpha$ and hence $\core^\alpha(G)= \core^\alpha(\core(G))$
for all countable groups $G$.

If $\alpha = \alpha_0+\omega$, we take a  
finitely generated group with $\depth(B)=\alpha_0$ and consider the quotient $E$ of $\DL\wr B$
constructed in our proof of Proposition \ref{p:omega+1}. This group has centre $\Z/3$ and
$\depth(E)=\alpha+1$.

If $\alpha$ is not of the form $\alpha_0+\omega$ then $\alpha$ is the supremum of 
the limit ordinals  $\beta<\alpha$ and we define $\wt{\G}^\alpha$ to be the amalgamated free product of
the groups $\Gt_\beta$ where the amalgamation identifies all of the centres $Z(\Gt_\beta)\cong\Z/3$.
Let $J_\beta$ be $\Gt_\beta$ modulo its centre and note that $\depth(J_\beta)=\beta$. We have a central extension
\[
1\to \Z/3\to \wt{\G}^\alpha \to \bigast_{\beta} J_\beta\to 1.
\]
For each $\beta$, the natural retraction   $\wt{\G}^\alpha\to\Gt_\beta$ maps $Z(\wt{\G}^\alpha)$ onto $Z(\Gt_\beta)$, hence
$Z(\wt{\G}^\alpha)\subseteq\core^{\beta}(\wt{\G}^{\alpha})$ for all $\beta<\alpha$,
 therefore $Z(\wt{\G}^\alpha)\subseteq \core^{\alpha}(\wt{\G})$.
On the other hand, $\wt{\G}^\alpha/Z(\wt{\G}^\alpha)$  is the free product of the $J_\beta$
and $\depth(J_\beta)=\beta$, so Lemma \ref{l:free-prod} tells us that the depth of $\wt{\G}^\alpha/Z(\wt{\G}^\alpha)$ 
is $\alpha=\omega\cdot\gamma$. 
Thus  $\core^\alpha(\wt{\G}^\gamma) = Z(\wt{\G}^\alpha) \cong  \Z/3$.  At this point, we have proved the weaker form of the
theorem with ``countable" in place of ``finitely generated" and the remainder of the proof simply involves adapting
the proof of Proposition \ref{p:3gen} so as to upgrade to finite generation. 

To this end, we let $\wt{\G}_0 = \wt{\G}^\alpha\ast_{\Z/3}((\Z/3)\times\Z)$, where the amalgamation 
identifies  $Z(\wt{\G}^\alpha)$ with $(\Z/3)\times 1$, and we fix  generators $\{\zeta, a_0, a_1\dots\}$ for $\wt{\G}_0$ where
$\zeta$ generates $Z(\wt{\G}^\alpha)$ and the $a_i$ have infinite order. We then take the HNN extension  
with infinitely many stable letters $t_i$, where $t_i$ conjugates $a_i$ to $a_0$ and fixes $\zeta$:
\[
\wt{\G}_1 = (\wt{\G}_0, t_1, t_2,\dots \mid t_i^{-1}a_it_i=a_0,\, [t_i,\zeta]=1\  \forall i).
\]
Note that $Z(\wt{\G}_1) = \<\zeta\>\cong\Z/3$, and $\wt{\G}_1$ modulo its centre is the group $\G_1$
we had constructed at this stage of the proof of Proposition \ref{p:3gen}.
Finally, we define $\Gta$ by taking a free group $F$ of rank $2$ and
amalgamating $\wt{\G}_1$ with $(\Z/3)\times F$ by
identifying $\<\zeta\>$ with $(\Z/3)\times 1$ and identifying an infinitely generated subgroup of 
$F$ with the free subgroup $\<t_1,t_2,\dots\>$.
By construction, $Z(\Gta) = \<\zeta\>\cong\Z/3$ and $\Gta/Z(\Gta)$
is the group $\G^+$ we had in proof of Proposition \ref{p:3gen}. In particular, 
$\depth(\Gta/Z(\Gta))=\alpha$. By construction, $\zeta\in\core^{\alpha}(\wt{\G}^\alpha)$, so
$\core(\Gta)=\<\zeta\>\cong\Z/3$ and $\depth(\Gta)= \alpha + 1$. This completes the induction. 
\end{proof}

A new idea is needed to answer the following question.

\begin{question} Which countable ordinals arise as the residual finiteness depths of finitely {\em presented} groups?
Might all limit ordinals less than $\omega^\omega$ and their successors arise in this way, for example?
\end{question}


\bigskip
\noindent{\textbf{Acknowledgements.}} I am grateful to the anonymous referees for their careful reading and insightful comments.    



\end{document}